\documentclass[preprint,12pt]{elsarticle}
\usepackage{amsmath,amssymb,amsthm,mathrsfs,dsfont}
\usepackage[margin=3cm]{geometry}
\usepackage{titlesec,hyperref}
\usepackage{color}
\usepackage{fancyhdr}
\titleformat{\subsection}{\it}{\thesubsection.\enspace}{1pt}{}
\makeatletter
\def\ps@pprintTitle{%
   \let\@oddhead\@empty
   \let\@evenhead\@empty
   \let\@oddfoot\@empty
   \let\@evenfoot\@oddfoot
}
\makeatother

\newtheorem{theo}{Theorem}[section]
\newtheorem{lemm}[theo]{Lemma}
\newtheorem{defi}[theo]{Definition}

\newtheorem{rema}[theo]{Remark}
\numberwithin{equation}{section}
\def\bel{\begin{equation}\label}

\def\eeq{\end{equation}}

\newcommand\R{{\mathbb R}\,}

\newcommand{\na}{{\nabla}}

\newcommand\tr{{\rm tr}\,}
\newcommand{\beq}{\begin{equation}}
\newcommand{\beno}{\begin{equation*}}
\newcommand{\eeno}{\end{equation*}}

\let\d=\delta

\let\lam=\lambda

\let\s=\sigma
\let\f=\frac
\let\vf=\varphi
\let\p=\psi

\let\G= \Gamma
\let\Lam=\Lambda

\let\Om=\Omega

\let\tri=\triangle
\let\ep=\epsilon
\def\na{\nabla}
\def\dive{\mathop{\rm div}\nolimits}
\def\exp{\mathop{\rm exp}\nolimits}

\def\tr{\mathop{\rm tr}\nolimits}
\def\D{\dot\Delta}

\allowdisplaybreaks

\begin{document}
\begin{frontmatter}
\title{Global well-posedness for the Phan-Thein-Tanner model in critical Besov spaces without damping}
\author[add1]{Yuhui Chen}\ead{chenyh339@mail.sysu.edu.cn}
\author[add2] {Wei Luo\corref{cor1}}\ead{luowei23@mail2.sysu.edu.cn}\cortext[cor1]{Corresponding author}
\author[add3]{Xiaoping Zhai}\ead{zhaixp@szu.edu.cn.}
\address[add1]{School of  Aeronautics and Astronautic, Sun Yat-sen University, Guangzhou, 510275, CHINA}
\address[add2]{School of Mathematics, Sun Yat-sen University, Guangzhou, 510275, CHINA}
\address[add3]{School  of Mathematics and Statistics, Shenzhen University,
 Shenzhen 518060, CHINA.}

\begin{abstract}
In this paper, we mainly investigate the Cauchy problem for the Phan-Thein-Tanner (PTT) model. The PPT model can be viewed as a Navier-Stokes equations couple with a nonlinear transport system. This model is derived from network theory for the polymeric fluid. We study about the global well posedness of the PTT model in critical Besov spaces. When the initial data is a small perturbation over around the equilibrium, we prove that the strong solution in critical Besov spaces exists globally.

\noindent {\it 2010 AMS Classification}: 35A01, 35B45, 35Q35, 76A05, 76D03.
\end{abstract}


\end{frontmatter}
\vspace*{10pt}

\section{Introduction}
In this paper, we consider the initial value problem for the following incompressible Phan-Thein-Tanner (PTT) model\cite{Thien1977,Thien1978}:
\begin{align}\label{OB}
\left\{
\begin{array}{ll}
u_{t}+u\cdot \na u-\mu \tri u+\na p=\mu_1\dive \tau, \quad (t, x) \in \mathbb{R}^{+}\times \mathbb{R}^3, \\[1ex]
\tau_t+u\cdot \na\tau+ (a+b\tr\tau)\tau+Q(\tau,\na u)=\mu_2 D(u),  \quad (t, x) \in \mathbb{R}^{+}\times \mathbb{R}^3, \\[1ex]
\dive u=0,  \quad (t, x) \in \mathbb{R}^{+}\times \mathbb{R}^3, \\[1ex]
u|_{t=0}=u_{0}(x),  \quad  \tau|_{t=0}=\tau_0(x), \quad x \in  \mathbb{R}^3. \\[1ex]
\end{array}
\right.
\end{align}
Here $u$ stands for the velocity and $p$ is the scalar pressure of fluid, $\tau$ is the stress tensor. $D(u)$ is the symmetric part of $\na u$, that is
\[D(u)=\f12(\na u+(\na u)^T).\]
$Q(\tau,\na u)$ is a given bilinear form
\[Q(\tau,\na u)=\tau \Om(u)-\Om(u)\tau+\lam(D(u)\tau+\tau D(u)),\]
where $\Om(u)$ is the skew-symmetric part of $\na u$, namely
\[\Om(u)=\f12(\na u-(\na u)^T).\]
$\mu>0$ is the viscosity coefficient and $\mu_1$ is the elastic coefficient. $a$ and $\mu_2$ are associated to the Debroah number $De=\frac{\mu_2}{a}$, which indicates the relation between the characteristic flow time and elastic time\cite{Bird1977}. $\lam\in[-1,1]$ is a physical parameter. In particular, we call the system co-rotational case when $\lam=0$. $b\geq 0$ is a constant relate to the rate of creation or destruction for the polymeric network junctions.

If $b=0$, the system \eqref{PTT} reduce to the famous Oldroyd-B model (See \cite{Oldroyd}) which has been studied widely. Let us review some mathematical results for the related Oldroyd type model. C. Guillop\'e and J.C. Saut \cite{Guillope1990-NA,Guillope1990} proved the existence of local strong solutions and the global existence of one dimensional shear flows. In \cite{Fernandez-Cara}, E. Fern\'andez-Cara, F. Guill\'en and R. Ortega studied the local well-posedness in Sobolev spaces. J. Chemin and N. Masmoudi \cite{Chemin2001} proved the local well-posedness in critical Besov spaces and give a low bound for the lifespan. In the co-rotational case, P. L. Lions and N. Masmoudi \cite{Lions-Masmoudi} proved the global existence of weak solutions. In \cite{Lin-Liu-Zhang2005}, F. Lin, C. Liu and P. Zhang proved that if the initial data is a small perturbation around equilibrium, then the strong solution is global in time. The similar results were obtained in several papers by virtue of different methods, see Z. Lei and Y. Zhou \cite{Lei-Zhou2005}, Z. Lei, C. Liu and Y. Zhou \cite{Lei2008}, T. Zhang and D. Fang \cite{Zhang-Fang2012}, Y. Zhu \cite{Zhu2018}. D. Fang, M. Hieber and R. Zi proved the global existence of strong solutions with a class of large data \cite{Fang-Zi2013,Fang-Zi2016}. Recently, Q. Chen and X. Hao \cite{Chen-Hao} and X. Zhai \cite{Zhai} study about the global well-posedness in the critical Besov spaces respectively. For the Oldroyd-B model, the global existence of strong solutions in two dimension without small conditions is still an open problem.

In this paper, we suppose that $b=\mu=\mu_1=\mu_2=1$ and $a=\lam=0$ in the PTT model. To our knowledge, there are a lot of numerical results about the PTT model (See, \cite{Oliveira1999,Mu-Zhao2012,Mu-Zhao2013,Bautista,Tamaddon-Jahromi}). However, there is no any well-posedness results about the PTT model. The nonlinear term $(\tr\tau) \tau$ in the PTT model will leads to some interesting phenomenon that is quiet different between the Oldroyd-B model. By virtue of the characteristic method, we prove that the strong solution of \eqref{PTT} will blow up in finite time when the initial data $\tr \tau_0<0$. This is a new phenomenon can not be founded in other viscoelastic model.

On the other hand, when $\tr \tau_0$ has a positive low bound $c_0$, we can prove the global existence of strong solution with small initial data. The idea is inspired by the method applied in \cite{Chen-Miao,Zhai}. The main different is to deal with the nonlinear term $(\tr\tau) \tau$. In \cite{Chen-Miao,Zhai}, the authors study about the following mixed linear system
\begin{align}\label{Lin}
\left\{
\begin{array}{ll}
u_{t}-\tri u-\Lam(\Lam^{-1}\mathbb{P}\dive \tau)=\mathbb P E, \\[1ex]
(\Lam^{-1}\mathbb P\dive\tau)_t+\Lam u=\Lam^{-1}\mathbb P\dive F, \\[1ex]
\end{array}
\right.
\end{align}
where $\mathbb {P}$ is the Leray projection operator and $\Lam^s=(-\tri)^{\f s2}$. Based on the above dissipative structure of $u$ and $\Lam^{-1}\mathbb{P} \dive \tau$, the authors in \cite{Chen-Hao,Zhai} prove the global existence of the strong solution for the Oldroyd-B model with small initial data in the critical Besov spaces. However, from the linearized system, we can not obtain any dissipation for $\tau$. Thus we can't control the nonlinear term $(\tr\tau) \tau$ in large time even tough the initial data is small. In order to deal with this difficult term, we need to change the original system into a new form.  Note that $u=0$ and \beno
\bar{\tau}=\f13\f1{\f1{c_0}+t}I=\begin{pmatrix} \f13\f1{\f1{c_0}+t} & 0 & 0 \\
 0 & \f13\f1{\f1{c_0}+t} & 0\\
 0 & 0& \f13\f1{\f1{c_0}+t}\\
 \end{pmatrix},
\eeno
is a special solution of \eqref{OB}. Let $\s=\tau-\bar{\tau}$, we rewrite \eqref{OB} in the perturbation form as:
\begin{align}\label{PTT}\tag{PTT}
\left\{
\begin{array}{ll}
u_{t}+u\cdot \na u- \tri u+\na p=\dive \s, \quad \dive u=0,\\[1ex]
\s_t+u\cdot \na\s+\f1{\f1{c_0}+t}(\s+\f13(\tr\s)I)+ (\tr\s)\s+Q(\s,\na u)= D(u),   \\[1ex]
u|_{t=0}=u_{0}(x),   \quad \s|_{t=0}=\tau_0(x)-\frac{c_0}{3}I. \\[1ex]
\end{array}
\right.
\end{align}

If $c_0>0$, we see that the linear term $\f1{\f1{c_0}+t} \s$ will leads to some dissipation information for $\s$.
Specifically, we will define the following basic energy in the low and high frequencies:
\beno
\mathcal E_1(t)=\|u^{\ell}\|_{\widetilde L^\infty_t(\dot B^{\f 12}_{2,1})}+\|\s^{\ell}\|_{\widetilde L^\infty_t(\dot B^{\f 12}_{2,1})},
\eeno
\beno
\mathcal E_2(t)=\|u^{\ell}\|_{ L^1_t(\dot B^{\f 52}_{2,1})}+\|\Lam^{-1}\mathbb P\dive\s ^{\ell}\|_{ L^1_t(\dot B^{\f 52}_{2,1})},
\eeno
\beno
\mathcal E_3(t)=\|u^h\|_{\widetilde L^\infty_t(\dot B^{\f 3p-1}_{p,1})}+\|\s^h\|_{\widetilde L^\infty_t(\dot B^{\f 3p}_{p,1})}+\|u^h\|_{ L^1_t(\dot B^{\f 3p+1}_{p,1})}+\|\Lam^{-1}\mathbb P\dive\s ^h\|_{ L^1_t(\dot B^{\f 3p}_{p,1})},
\eeno
\beno
\mathcal E_4(t)=\|\tr\s^{\ell}\|_{L_t^1(\dot B^{\f 32}_{2,1})}+\|\tr\s^h\|_{L_t^1(\dot B^{\f 3p}_{p,1})},
\eeno
and then estimate the above terms one by one. By virtue of the Littlewood-Paley theory, we deduce that
\begin{align*}
E(t)=&\mathcal E_1(t)+\mathcal E_2(t)+\mathcal E_3(t)+\mathcal E_4(t)\\
\leq &C^*\Big[1+\exp\big(\mathcal E_2(t)+\mathcal E_3(t)+\mathcal E_4(t)\big)\Big]\mathcal E(0)+\Big[\mathcal E_1^2(t)+\mathcal E_2^2(t)+\mathcal E_3^2(t)+\mathcal E_4^2(t)\Big].
\end{align*}

From the above estimate we can obtain the global existence by a standard continuous argument under the condition $u_0$ and $\s_0$ is small enough.

A functional space is called critical if the associated norm is invariant under the scaling transformation. Although the system \eqref{OB}
 does not have any scaling invariance, one may find that $(u,\tau)$ is a solution of \eqref{Lin}, then
 \[(u_\lambda, \tau_\lambda, p_\lambda)=(\lambda u(\lambda^2t,\lambda x), \tau(\lambda^2t,\lambda x),\lambda^2p(\lambda^2t,\lambda x))\]
is also a solution of \eqref{Lin}. Thus, we can use the linearized system to define the critical spaces. The reason to consider the well-posedness in critical spaces has been fully explained in \cite{Chemin2001}. There is a lot of papers study about the well-posedness in critical Besov spaces, one can refers to \cite{Chen-Miao,Chen-Zhang,Chen-Miao-Zhang,Danchin2000,Danchin2001,Danchin-He} and references therein.

{\bf Notation. Since all function spaces in through out the paper are over $\mathbb{R}^3$, for simplicity, we drop $\mathbb{R}^3$ in the notation of function spaces if there is no ambiguity.  $A\lesssim B$ stands for $A\leq C B$ for some constant $C>0$ independent of $A$ and $B$. $\mathcal S'(\mathbb R^3)$ denotes the set of tempered distributions.  For any $z\in \mathcal S'(\mathbb R^3)$, the lower and higher
oscillation parts can be expressed as
\[z^\ell=\sum_{j\leq N} \D_j z, \quad \quad z^h= \sum_{j> N} \D_j z,\]
where $\D_j$ are the Littlewood-Paley dyadic blocks and $N$ is a large but fixed integer.
}

Our main result can be stated as follow:

\begin{theo}\label{them}
 Let $0<c_0<+\infty$. Suppose that $\dive u_0=0$, $(\s_0)_{ij}=(\s_0)_{ji}$, and the initial data $(u_0^{\ell}, \s_0^{\ell})\in \dot B^{\f 12}_{2,1}$, $u_0^h\in \dot B^{\f 3p-1}_{p,1}$, $\s_0^h\in \dot B^{\f 3p}_{p,1}$ with $p\in[2,4]$.  There exists a $\ep_0$ such that if
 \[0<\delta_0:=\|(u_0^{\ell},\s_0^{\ell})\|_{\dot B^{\f 12}_{2,1}}+\|u_0^h\|_{\dot B^{\f 3p-1}_{p,1}}+\|\s_0^h\|_{\dot B^{\f 3p}_{p,1}}\leq \ep_0,\]
 then the problem \eqref{PTT} admits a unique global solution $(u(t), \s(t))$ satisfying that for all $t\geq0$: \beno
\begin{split}
&\|u^{\ell}\|_{\widetilde L^\infty_t(\dot B^{\f 12}_{2,1})}+\|\s^{\ell}\|_{\widetilde L^\infty_t(\dot B^{\f 12}_{2,1})}
+\|u^h\|_{\widetilde L^\infty_t(\dot B^{\f 3p-1}_{p,1})}+\|\s^h\|_{\widetilde L^\infty_t(\dot B^{\f 3p}_{p,1})}\\
&\quad+\|u^{\ell}\|_{ L^1_t(\dot B^{\f 52}_{2,1})}+\|\Lam^{-1}\mathbb P\dive\s ^{\ell}\|_{ L^1_t(\dot B^{\f 52}_{2,1})}
+\|u^h\|_{ L^1_t(\dot B^{\f 3p+1}_{p,1})}+\|\Lam^{-1}\mathbb P\dive\s ^h\|_{ L^1_t(\dot B^{\f 3p}_{p,1})}
\leq C(c_0)\d_0,
\end{split}
\eeno
where $C>0$ is a positive constant independent of $t$.
\end{theo}

The remainder of the paper is organized as follows. In Section 2 we review some basic proposition about the homogeneous Besov space. In Section 3 devote to study about the global existence of the strong solution with small initial data.

\section{Preliminary}
In this section, we first give the definition of the homogeneous Besov space. (See \cite{Bahouri2011} for more details)
Let $\mathcal{C}$ be the annulus $\{\xi\in\mathbb{R}^{3}\big|\frac{3}{4}\leq|\xi|\leq\frac{8}{3}\}.$ There exists radial function $\varphi$, valued in the interval $[0,1]$, such that
\begin{align}
\forall\xi\in\mathbb{R}^{3}\backslash\{0\},~\sum_{j\in\mathbb{Z}}\varphi(2^{-j}\xi)=1,
\end{align}
\begin{align}
|j-j'|\geq2\Rightarrow Supp ~\varphi(2^{-j}\xi)\cap Supp ~\varphi(2^{-j'}\xi)=\emptyset.
\end{align}
 The homogeneous dyadic blocks $\dot{\Delta}_{j}$ are defined by
\begin{align} \dot{\Delta}_{j}u=\varphi(2^{-j}D)u=\int_{\mathbb{R}^{3}}h(2^{j}y)u(x-y)dy,
\end{align}
\begin{align}
\dot{S}_{j}u=\chi(2^{-j}D)u=\int_{\mathbb{R}^{3}}\widetilde{h}(2^{j}y)u(x-y)dy,
\end{align}
where
$$h(x)=\mathcal{F}^{-1}(\varphi)(x),\quad \widetilde{h}(x)=\mathcal{F}^{-1}(\chi)(x), \quad \chi(\xi)=1-\sum_{j\geq 0}\varphi(2^{-j}\xi).$$
We denote by $S'_{h}(\R^3)$ the space of tempered distributions $u$ such that\beno
\lim_{\lam\to\infty}\|\theta(\lam D)u\|_{L^\infty}=0,\quad \forall \theta\in\mathcal D(\R^3).
\eeno
The homogeneous Besov space is denoted by $\dot{B}^{s}_{p,r}$, that is
$$\dot{B}^{s}_{p,r}=\Big\{u\in S'_{h}\Big{|}\|u\|_{\dot{B}^{s}_{p,r}}=\big\|2^{js}\|\dot{\Delta}_{j}u\|_{L^{p}_x}\big\|_{l^r}<\infty\Big\},$$
where $s\in\mathbb{R}$ and $p, r\in[1,+\infty]$. One can easily check that
\[\|u^\ell\|_{\dot{B}^{s}_{p,1}}\approx\sum_{j\leq N}2^{js}\|\D_j u\|_{L^p}, \quad \text{and} \quad \|u^h\|_{\dot{B}^{s}_{p,1}}\approx\sum_{j> N}2^{js}\|\D_j u\|_{L^p}.\]

In order to study the product acts on Besov spaces,  we need to use the Bony decomposition.
\begin{defi}
\cite{Bahouri2011} For functions $u$ and $v$, Bony's decomposition in the homogeneous context is defined by
$$uv=\dot{T}_{u}v+\dot{R}(u,v)+\dot{T}_{v}u,$$
where
$$\dot{T}_{u}v\triangleq\sum_{j}\dot{S}_{j-1}u\dot{\Delta}_{j}v,\quad \dot{R}(u,v)\triangleq\sum_{|k-j|\leq1}\dot{\Delta}_{k}u\dot{\Delta}_{j}v.$$
\end{defi}

The following two lemmas will be used in the sequel.
\begin{lemm}\label{com}
Let $p\in [2,\infty)$, for any $v^{\ell}\in{\dot B^{\f 12}_{2,1}}$, $v^h\in{\dot B^{\f 3p-1}_{p,1}}$, $w^{\ell}\in{\dot B^{\f 32}_{2,1}}$, $w^h\in{\dot B^{\f 3p}_{p,1}}$ and $\na u^{\ell}\in{\dot B^{\f 32}_{2,1}}$, $\na u^h\in{\dot B^{\f 3p}_{p,1}}$, then we have
\beno
\begin{split}
\sum_{j\leq N}2^{\f12j}\|[u\cdot\na,\D_j]v\|_{L^2}
&\lesssim\|v^{\ell}\|_{\dot B^{\f 12}_{2,1}}\|\na u\|_{\dot B^{\f 3p}_{p,1}}+\|v\|_{\dot B^{\f 3p-1}_{p,1}}\|\na u^{\ell}\|_{\dot B^{\f 32}_{2,1}},\\
\sum_{j> N}2^{(\f 3p-1)j}\|[u\cdot\na,\D_j]v\|_{L^p}
&\lesssim\|v^h\|_{\dot B^{\f 3p-1}_{p,1}}\|\na u\|_{\dot B^{\f 3p}_{p,1}}+\|v\|_{\dot B^{\f 3p-1}_{p,1}}\|\na u^h\|_{\dot B^{\f 3p}_{p,1}},\\
\sum_{j\leq N}2^{\f 32j}\|[u\cdot\na,\D_j]w\|_{L^2}
&\lesssim\|w^{\ell}\|_{\dot B^{\f 32}_{2,1}}\|\na u\|_{\dot B^{\f 3p}_{p,1}}+\|w\|_{\dot B^{\f 3p}_{p,1}}\|\na u^{\ell}\|_{\dot B^{\f 32}_{2,1}},\\
\sum_{j> N}2^{\f 3pj}\|[u\cdot\na,\D_j]w\|_{L^p}
&\lesssim\|w^h\|_{\dot B^{\f 3p}_{p,1}}\|\na u\|_{\dot B^{\f 3p}_{p,1}}+\|w\|_{\dot B^{\f 3p}_{p,1}}\|\na u^h\|_{\dot B^{\f 3p}_{p,1}}.
\end{split}
\eeno
\end{lemm}
\begin{proof}
We only deal with the first inequality. Using Bony's decomposition for $[u\cdot\na,\D_j]v$, then we have\beno
[u\cdot\na,\D_j]v
=\sum_{|i-j|\leq4}[\dot S_{i-1}u\cdot\na,\D_j]\D_iv+\sum_{i\geq j-3}[\D_iu\cdot\na,\D_j]\dot S_{i-1}v+\sum_{|i-j|\leq 3}[\D_iu\cdot\na,\D_j]\tilde \D_iv=I_1+I_2+I_3,
\eeno
where $\tilde \D_i=\D_{i-1}+\D_i+\D_{i+1}.$
By the H\"{o}lder inequality, we get\beno
\begin{split}
\|I_1\|_{L^2}+\|I_3\|_{L^2}
\lesssim &\sum_{|i-j|\leq4}2^{-j}(\|\na \dot S_{i-1}u\|_{L^\infty}\|\na\D_i v\|_{L^2}+\|\na\D_i u\|_{L^\infty}\|\na \tilde \D_iv\|_{L^2})\\
\lesssim& \sum_{|i-j|\leq4}2^{i-j}\|\na u\|_{\dot B^{\f 3p}_{p,1}}\|\D_j v\|_{L^2}
\lesssim \|\na u\|_{\dot B^{\f 3p}_{p,1}}\|\D_j v\|_{L^2},
\end{split}
\eeno
which give rise to\beno
\begin{split}
\sum_{j\leq N}2^{\f12j}(\|I_1\|_{L^2}+\|I_3\|_{L^2})
\lesssim\|\na u\|_{\dot B^{\f 3p}_{p,1}}\|v^{\ell}\|_{\dot B^{\f12}_{2,1}}.
\end{split}
\eeno
As for $I_2$, we split it into two terms\beno
I_2=\sum_{|i-j|\leq4}[\D_iu\cdot\na,\D_j]\dot S_{i-1}v+\sum_{i\geq j+3} \D_iu\cdot\na(\D_j\dot S_{i-1}v)
=I_{21}+I_{22},
\eeno
by the H\"{o}lder inequality, we get\beno
\begin{split}
\|I_2\|_{L^2}
\lesssim& \sum_{|i-j|\leq4}2^{-j}(\|\na \D_i u\|_{L^2}\|\na \dot S_{i-1}v\|_{L^\infty}+\sum_{i\geq j+3} \|\D_iu\cdot\na(\D_j\dot S_{i-1}v)\|_{L^2}\\
\lesssim& \sum_{\substack{|i-j|\leq4}}2^{2i-j}\|\na\D_i u\|_{L^2}\|\Lam^{-1} v\|_{L^\infty}+\sum_{i\geq j+3} 2^{j-i}\|\na u\|_{L^\infty}\|\D_j\dot S_{i-1}v\|_{L^2}\\
\lesssim& 2^{j}\|\na\D_j u\|_{L^2}\|\Lam^{-1} v\|_{L^\infty}+\|\na u\|_{L^\infty}\|\D_j v\|_{L^2},\\
\end{split}
\eeno
which gives rise to\beno
\begin{split}
\sum_{j\leq N}2^{\f12j}\|I_{2}\|_{L^2}
\lesssim &\sum_{\substack{j\leq N}}2^{\f 32j}\|\na \D_j u\|_{L^2}\|\Lam^{-1} v\|_{L^\infty}+\sum_{j\leq N}2^{\f12j}\|\na u\|_{L^\infty}\|\D_j v\|_{L^2}\\
\lesssim&\|\na u^{\ell}\|_{\dot B^{\f 32}_{2,1}}\|v\|_{\dot B^{\f 3p-1}_{p,1}}+\|\na u\|_{\dot B^{\f 3p}_{p,1}}\|v^{\ell}\|_{\dot B^{\f 12}_{2,1}}.
\end{split}
\eeno
Together with the above estimates, then we prove the Lemma.
\end{proof}

\begin{lemm}\label{mult}
{\sl Let $p\in [2,4]$, for any $v^{\ell}\in{\dot B^{\f 12}_{2,1}}$, $v^h\in{\dot B^{\f 3p-1}_{p,1}}$, $w^{\ell}\in{\dot B^{\f 32}_{2,1}}$, $w^h\in{\dot B^{\f 3p}_{p,1}}$,  $u^{\ell}\in{\dot B^{\f 32}_{2,1}}$, $u^h\in{\dot B^{\f 3p}_{p,1}}$ then we have
\beno
\begin{split}
\|(vu)^{\ell}\|_{\dot B^{\f 12}_{2,1}}
\lesssim&(\|v^{\ell}\|_{\dot B^{\f 12}_{2,1}}+\|v^h\|_{\dot B^{\f 3p-1}_{p,1}})\|u\|_{\dot B^{\f 3p}_{p,1}},\\
\|(vu)^h\|_{\dot B^{\f 3p-1}_{p,1}}
\lesssim&(\|v^{\ell}\|_{\dot B^{\f 12}_{2,1}}+\|v^h\|_{\dot B^{\f 3p-1}_{p,1}})\|u\|_{\dot B^{\f 3p}_{p,1}},\\
\|(wu)^{\ell}\|_{\dot B^{\f 32}_{2,1}}
\lesssim&(\|w^{\ell}\|_{\dot B^{\f 32}_{2,1}}+\|w^h\|_{\dot B^{\f 3p}_{p,1}})(\|u^{\ell}\|_{\dot B^{\f 32}_{2,1}}+\|u^h\|_{\dot B^{\f 3p}_{p,1}}),\\
\|(wu)^h\|_{\dot B^{\f 3p}_{p,1}}
\lesssim&\|w\|_{\dot B^{\f 3p}_{p,1}}\|u\|_{\dot B^{\f 3p}_{p,1}}.
\end{split}
\eeno
}
\end{lemm}

\begin{proof}
Using the Bony decomposition, then we get\beno
\dot S_{N+1}(vu)=\dot S_{N+1}(\dot T_v u+\dot R(v, u))+\dot T_{ u}\dot S_{N+1}v+[\dot S_{N+1},\dot T_u] v.
\eeno
 Let $\f1q=\f12-\f1p$, $q\geq p$, one can check that $\f 3q-1<0$. By the definition of $\dot T_v u$ and $\dot R(v, u)$, we have\beno
\begin{split}
\|\dot S_{N+1}(\dot T_v u+\dot R(v, u))\|_{\dot B^{\f 12}_{2,1}}
\lesssim &\sum_{\substack{k\geq j-2}}2^{\f12j}\|\D_{j}(\dot S_{k+2}v\D_k u)\|_{L^2}\\
\lesssim &\sum_{\substack{k\geq j-2}}2^{(\f 3p+\f 3q-1)j}\|\dot S_{k+2}v\|_{L^q}\|\D_k u\|_{L^p}\\
\lesssim &\|v\|_{\dot B^{\f 3q-1}_{q,1}}\|u\|_{\dot B^{\f 3p}_{p,1}}
\lesssim \|v\|_{\dot B^{\f 3p-1}_{p,1}}\|u\|_{\dot B^{\f 3p}_{p,1}},
\end{split}
\eeno
and\beno
\|\dot T_{u}\dot S_{N+1}v\|_{\dot B^{\f 12}_{2,1}}
\lesssim\|\dot S_{N+1}v\|_{\dot B^{\f12}_{2,1}}\|u\|_{L^\infty}
\lesssim \|v^{\ell}\|_{\dot B^{\f 12}_{2,1}}\|u\|_{\dot B^{\f 3p}_{p,1}},
\eeno
for the last term, we have\beno
\|[\dot S_{N+1},\dot T_u] v\|_{\dot B^{\f 12}_{2,1}}
\lesssim\|v\|_{\dot B^{\f 3p-1}_{p,1}}\|\na u\|_{\dot B^{\f 3q-1}_{q,1}}
\lesssim\|v\|_{\dot B^{\f 3p-1}_{p,1}}\|u\|_{\dot B^{\f 3p}_{p,1}}.
\eeno
Taking advantage of the Bony decomposition again, we obtain\beno
(I-\dot S_{N+1})(vu)=(I-\dot S_{N+1})(\dot T_v u+\dot R(v, u))+(I-\dot S_{N+1})\dot T_{ u}v.
\eeno
Notice that $\f 3q-1<0$. Let $\f1p=\f1q+\f1r$, then we have\beno
\begin{split}
&\|(I-\dot S_{N+1})(\dot T_v u+\dot R(v, u))\|_{\dot B^{\f 3p-1}_{p,1}}\\
\lesssim &\sum_{\substack{k\geq j-2}}2^{(\f 3p-1)j}\|\D_{j}(\dot S_{k+2}v\D_k u)\|_{L^p}+\sum_{\substack{k\geq j-2}}2^{(\f 3p-1)j}\|\D_{j}(\D_kv \dot S_{k-1}u)\|_{L^p}\\
\lesssim &\sum_{\substack{k\geq j-2}}2^{(\f 3q+\f 3r-1)j}\|\dot S_{k+2}  v\|_{L^q}\| \D_k u\|_{L^r}
+\sum_{\substack{k\geq j-2}}2^{(\f 3p-1)j}\|\D_kv \|_{L^p}\|\dot S_{k-1} u\|_{L^\infty}\\
\lesssim &\|v\|_{\dot B^{\f 3q-1}_{q,1}}\|u\|_{\dot B^{\f 3r}_{r,1}}+\|v^h\|_{\dot B^{\f 3p-1}_{p,1}}\|u\|_{L^\infty}\\
\lesssim &(\|v^{\ell}\|_{\dot B^{\f 12}_{2,1}}+\|v^h\|_{\dot B^{\f 3p-1}_{p,1}})\|u\|_{\dot B^{\f 3p}_{p,1}},
\end{split}
\eeno
and \beno
\begin{split}
\|(I-\dot S_{N+1})\dot T_u v\|_{\dot B^{\f 3p-1}_{p,1}}
\lesssim\|v\|_{\dot B^{\f 3p-1}_{p,1}}\|u\|_{L^\infty}
\lesssim\|v\|_{\dot B^{\f 3p-1}_{p,1}}\|u\|_{\dot B^{\f 3p}_{p,1}}.
\end{split}
\eeno
Let $\widetilde{T}_w u =\dot T_w u+\dot R(w, u)$. By a similar computation, we have \beno
\begin{split}
\|\|\dot S_{N+1}(wu)\|_{\dot B^{\f 32}_{2,1}}
\leq&\|\dot S_{N+1}(\widetilde{T}_u w)\|_{\dot B^{\f 32}_{2,1}}
+\|\dot S_{N+1}(\dot{T}_w u)\|_{\dot B^{\f 32}_{2,1}}\\
\lesssim&\|[\dot S_{N+1},\widetilde{T}_w] u\|_{\dot B^{\f 32}_{2,1}}+\|\widetilde{T}_w( \dot S_{N+1}u)\|_{\dot B^{\f 32}_{2,1}}
+\|[\dot S_{N+1},\dot{T}_u] w\|_{\dot B^{\f 32}_{2,1}}+\|\dot{T}_u( \dot S_{N+1}w)\|_{\dot B^{\f 32}_{2,1}}\\
\lesssim&\|w\|_{\dot B^{\f 3p}_{p,1}}\|\na u\|_{\dot B^{\f 3q-1}_{q,1}}
+\|w\|_{L^\infty}\|u^{\ell}\|_{\dot B^{\f 32}_{2,1}}+
\|u\|_{\dot B^{\f 3p}_{p,1}}\|\na w\|_{\dot B^{\f 3q-1}_{q,1}}
+\|u\|_{L^\infty}\|w^{\ell}\|_{\dot B^{\f 32}_{2,1}}\\
\lesssim&(\|w^{\ell}\|_{\dot B^{\f 32}_{2,1}}+\|w^h\|_{\dot B^{\f 3p}_{p,1}})(\|u^{\ell}\|_{\dot B^{\f 32}_{2,1}}+\|u^h\|_{\dot B^{\f 3p}_{p,1}}),
\end{split}
\eeno
and
\beno
\begin{split}
\|(I-\dot S_{N+1})(wu)\|_{\dot B^{\f 3p}_{p,1}}
\lesssim\|w\|_{\dot B^{\f 3p}_{p,1}}\|u\|_{L^\infty}+\|u\|_{\dot B^{\f 3p}_{p,1}}\|w\|_{L^\infty}
\lesssim\|w\|_{\dot B^{\f 3p}_{p,1}}\|u\|_{\dot B^{\f 3p}_{p,1}}.
\end{split}
\eeno
Hence we prove the Lemma.
\end{proof}

\section{Global existence}
In this section, we are going to prove our main result. There is no derivative in the additional term in \eqref{PTT}, the proof of local well-posedness for \eqref{PTT} is similar to the Oldroyd-B model(See \cite{Fernandez-Cara,Lin-Liu-Zhang2005, Zhai}) and we omit the detail here. In order to prove the global existence of strong solutions, we give the some basic energies as follows,
\beq
\mathcal E(0)=\|u^{\ell}_0\|_{\dot B^{\f 12}_{2,1}}+\|\s^{\ell}_0\|_{\dot B^{\f 12}_{2,1}}
+\|u^h_0\|_{\dot B^{\f 3p-1}_{p,1}}+\|\s^h_0\|_{\dot B^{\f 3p}_{p,1}},
\eeq
\beq
\mathcal E_1(t)=\|u^{\ell}\|_{\widetilde L^\infty_t(\dot B^{\f 12}_{2,1})}+\|\s^{\ell}\|_{\widetilde L^\infty_t(\dot B^{\f 12}_{2,1})},
\eeq
\beq
\mathcal E_2(t)=\|u^{\ell}\|_{ L^1_t(\dot B^{\f 52}_{2,1})}+\|\Lam^{-1}\mathbb P\dive\s ^{\ell}\|_{ L^1_t(\dot B^{\f 52}_{2,1})},
\eeq
\beq
\mathcal E_3(t)=\|u^h\|_{\widetilde L^\infty_t(\dot B^{\f 3p-1}_{p,1})}+\|\s^h\|_{\widetilde L^\infty_t(\dot B^{\f 3p}_{p,1})}+\|u^h\|_{ L^1_t(\dot B^{\f 3p+1}_{p,1})}+\|\Lam^{-1}\mathbb P\dive\s ^h\|_{ L^1_t(\dot B^{\f 3p}_{p,1})},
\eeq
\beq
\mathcal E_4(t)=\|\tr\s^{\ell}\|_{L_t^1(\dot B^{\f 32}_{2,1})}+\|\tr\s^h\|_{L_t^1(\dot B^{\f 3p}_{p,1})},
\eeq
where $\mathbb P=\mathbb I-\tri^{-1}\na\dive$ is the Leray projection operator.
We shall derive the a priori estimates of $\mathcal E_1(t)$, $\mathcal E_2(t)$, $\mathcal E_3(t)$ and $\mathcal E_4(t)$ respectively.

\subsection{The estimates of $\mathcal E_1(t)$}
Applying the operator $\D_j\mathbb P$ to the first equation of \eqref{PTT} and $\D_j$ to the second equation of \eqref{PTT}, we have
\begin{align}\label{ob1}
\left\{
\begin{array}{ll}
(\D_j u)_{t}+u\cdot \na \D_ju-\tri \D_ju=\D_j\mathbb P\dive \s+[u\cdot\na,\D_j\mathbb P]u, \\[1ex]
(\D_j\s)_t+u\cdot \na\D_j \s+\f1{\f1{c_0}+t}(\D_j\s+\f13(\tr\D_j\s)I)\\[1ex]
\qquad=\D_j D(u)-\D_j  ((\tr\s)\s+Q(\s,\na u))+[u\cdot\na,\D_j]\s. \\[1ex]
\end{array}
\right.
\end{align}
Notice that $\dive u=0$. Taking the $L^2$ scalar product of the first equation of \eqref{ob1} with $\D_j u$ and the second equation of \eqref{ob1} with $\D_j \s$, then we obtain that
\bel{11}
\begin{split}
&\f12\f{d}{dt}(\|\D_j u\|_{L^2}^2+\|\D_j \s\|_{L^2}^2)+\|\na\D_j u\|_{L^2}^2+\f23\f1{\f1{c_0}+t}\|\D_j\s
\|_{L^2}^2\\
\leq&\int_{\R^3} (\D_j\mathbb P\dive \s\cdot\D_j u+\D_j D(u)\cdot\D_j \s) dx\\
&\quad+\int_{\R^3} ([u\cdot\na,\D_j\mathbb P]u\cdot\D_j u+[u\cdot\na,\D_j]\s\cdot\D_j \s) dx\\
&\quad
-\int_{\R^3} \D_j  ((\tr\s)\s+Q(\s,\na u))\cdot\D_j \s dx.
\end{split}
\eeq
Since $\s_{ij}=\s_{ji}$, it follows that\beno
\int_{\R^3}(\D_j\mathbb P\dive \s\cdot\D_j u+\D_j D(u)\cdot\D_j \s) dx=0.
\eeno
Integrating in time and multiplying both sides of \eqref{11} by $2^{\f 12j}$, summing up about $j\leq N$, then we obtain that,
\beno
\begin{split}
&\|u^{\ell}\|_{\widetilde L^\infty_t(\dot B^{\f 12}_{2,1})}+\|\s^{\ell}\|_{\widetilde L^\infty_t(\dot B^{\f 12}_{2,1})}
\lesssim \|u^{\ell}_0\|_{\dot B^{\f 12}_{2,1}}+\|\s^{\ell}_0\|_{\dot B^{\f 12}_{2,1}}
+\sum_{j\leq N}2^{\f 12j}\|[u\cdot\na,\D_j\mathbb P]u\|_{L^1_t(L^2)}\\
&\quad+\sum_{j\leq N}2^{\f 12j}\|[u\cdot\na,\D_j]\s\|_{L^1_t(L^2)}
+\sum_{j\leq N}2^{\f 12j}\|\D_j  ((\tr\s)\s+Q(\s,\na u))\|_{L^1_t(L^2)}.
\end{split}
\eeno
Applying Lemma \ref{com} yields that\beno
\begin{split}
&\sum_{j\leq N}2^{\f12j}(\|[u\cdot\na,\D_j\mathbb P]u\|_{L^1_t(L^2)}+\|[u\cdot\na,\D_j]\s\|_{L^1_t(L^2)})\\
\lesssim &\int_0^t\big(\|u^{\ell}\|_{\dot B^{\f 12}_{2,1}}+\|u^h\|_{\dot B^{\f 3p-1}_{p,1}}+\|\s^{\ell}\|_{\dot B^{\f 12}_{2,1}}+\|\s^h\|_{\dot B^{\f 3p-1}_{p,1}}\big)\big(\|\na u^{\ell}\|_{\dot B^{\f 32}_{2,1}}+\|\na u^h\|_{\dot B^{\f 3p}_{p,1}}\big)ds.
\end{split}
\eeno
Taking advantage of Lemma \ref{mult}, we have
\beno
\begin{split}
&\sum_{j\leq N}2^{\f12j}\|\D_j  ((\tr\s)\s+Q(\s,\na u))\|_{L^1_t(L^2)}\\
\lesssim &\int_0^t\big(\|\s^{\ell}\|_{\dot B^{\f 12}_{2,1}}+\|\s^h\|_{\dot B^{\f 3p-1}_{p,1}}\big)\big(\|\tr\s^{\ell}\|_{\dot B^{\f 32}_{2,1}}+\|\tr\s^h\|_{\dot B^{\f 3p}_{p,1}}+\|\na u^{\ell}\|_{\dot B^{\f 32}_{2,1}}+\|\na u^h\|_{\dot B^{\f 3p}_{p,1}}\big)ds.
\end{split}
\eeno
According to above estimates, we deduce that\beq\label{e1}
\begin{split}
&\|u^{\ell}\|_{\widetilde L^\infty_t(\dot B^{\f 12}_{2,1})}+\|\s^{\ell}\|_{\widetilde L^\infty_t(\dot B^{\f 12}_{2,1})}\\
\lesssim &\mathcal E(0)+\int_0^t\big(\|u^{\ell}\|_{\dot B^{\f 12}_{2,1}}+\|u^h\|_{\dot B^{\f 3p-1}_{p,1}}
+\|\s^{\ell}\|_{\dot B^{\f 12}_{2,1}}+\|\s^h\|_{\dot B^{\f 3p}_{p,1}}\big)\\
&\quad\times\big(\|\tr\s^{\ell}\|_{\dot B^{\f 32}_{2,1}}+\|\tr\s^h\|_{\dot B^{\f 3p}_{p,1}}+\|\na u^{\ell}\|_{\dot B^{\f 32}_{2,1}}+\|\na u^h\|_{\dot B^{\f 3p}_{p,1}}\big)ds,
\end{split}
\eeq
which implies\beq
\mathcal E_1(t)
\lesssim \mathcal E(0)+\big(\mathcal E_1(t)+\mathcal E_3(t)\big)\big(\mathcal E_2(t)+\mathcal E_3(t)+\mathcal E_4(t)\big).
\eeq

\subsection{The estimates of $\mathcal E_2(t)$}
Applying the operator $\mathbb P$ to the first equation of \eqref{PTT} and  $\Lam^{-1}\mathbb P\dive$ to the second equation of \eqref{PTT}, we have
\begin{align}\label{ob2}
\left\{
\begin{array}{ll}
u_{t}+\mathbb P(u\cdot \na u)-\tri u=\mathbb P\dive \s, \\[1ex]
(\Lam^{-1}\mathbb P\dive\s)_t+\Lam^{-1}\mathbb P\dive(u\cdot \na\s)+\f1{\f1{c_0}+t}\Lam^{-1}\mathbb P\dive\s+\f12\Lam u\\[1ex]
\qquad=- \Lam^{-1}\mathbb P\dive((\tr\s)\s+Q(\tau,\na u)). \\[1ex]
\end{array}
\right.
\end{align}
Applying $\D_j$ to the system \eqref{ob2}, then we obtain the following system:
\begin{align}\label{ob3}
\left\{
\begin{array}{ll}
\D_ju_{t}+u\cdot \na \D_ju-\tri \D_ju-\Lam\D_j\p=f_j,\\[1ex]
\D_j \p_t+u\cdot \na \D_j\p+\f1{\f1{c_0}+t}\D_j \p+\f12\Lam \D_j u=g_j, \\[1ex]
\end{array}
\right.
\end{align}
where\beno
\begin{split}
\p&=\Lam^{-1}\mathbb P\dive\s, \quad f_j=[u\cdot\na, \D_j\mathbb P]u,\\
g_j&=[u\cdot\na, \D_j \Lam^{-1}\mathbb P\dive]\s-\D_j \Lam^{-1}\mathbb P\dive((\tr\s)\s+Q(\s,\na u)).
\end{split}
\eeno
Let $0<\eta<1$ be a small constant which will be determined later on. Taking inner product with $(1-\eta)\D_j u$ for the first equation of \eqref{ob3}, and $\D_j\p$ for the second equation of \eqref{ob3}, and then we have\beno
\begin{split}
&\f12\f{d}{dt}\big((1-\eta)\|\D_ju\|_{L^2}^2+\|\D_j\p\|_{L^2}^2\big)+(1-\eta)\|\na\D_ju\|_{L^2}^2
+\eta \int_{\mathbb R^3} \Lam \D_j u\cdot \D_j\p dx\\
\lesssim &\|f_j\|_{L^2}\|\D_ju\|_{L^2}+\|g_j\|_{L^2}\|\D_j\p\|_{L^2}.
\end{split}
\eeno
Denote $\vf=2\Lam\p-u$, we have the following equation:
\beno
\D_j\vf_{t}+u\cdot \na \D_j\vf+(1+\f2{\f1{c_0}+t})\Lam\D_j\p=2\Lam g_j-f_j+2[u\cdot \na,\Lam]\D_j\p,
\eeno
taking $L^2$ inner product of $\D_j \vf$,  then we have\beno
\begin{split}
&\f12\f{d}{dt}\|\D_j\vf\|_{L^2}^2+2(1+\f2{\f1{c_0}+t})\|\Lam\D_j\p\|_{L^2}^2-(1+\f2{\f1{c_0}+t})\int_{\mathbb R^3} \Lam\D_j\p \cdot \D_j  u dx\\
\lesssim &\big(\|f_j\|_{L^2}+2^j\|g_j\|_{L^2}+\|[u\cdot \na,\Lam]\D_j\p\|_{L^2}\big)\big(\|\Lam\D_j\p\|_{L^2}+\|\D_j u\|_{L^2}\big).
\end{split}
\eeno
Together with the above inequalities and using the fact that $1\leq (1+\f2{\f1{c_0}+t})\leq 1+2c_0$, we deduce that\beno
\begin{split}
&\f12\f{d}{dt}\big((1-\eta)\|\D_ju\|_{L^2}^2+\|\D_j\p\|_{L^2}^2+\eta\|\D_j\vf\|_{L^2}^2\big)+(1-\eta)\|\na\D_ju\|_{L^2}^2 +\eta\|\Lam\D_j\p\|_{L^2}^2\\
\lesssim&
\big(\|f_j\|_{L^2}+(1+2^j)\|g_j\|_{L^2}+\|[u\cdot \na,\Lam]\D_j\p\|_{L^2}\big)\big(\|\D_ju\|_{L^2}+(1+2^j)\|\D_j\p\|_{L^2}\big).
\end{split}
\eeno
For any $j\leq N$, we can find a $\eta=\eta(N)>0$ small enough such that
\[(1-\eta)\|\D_ju\|_{L^2}^2+\|\D_j\p\|_{L^2}^2+\eta\|\D_j\vf\|_{L^2}^2\geq C_N(\|\D_ju\|_{L^2}^2+ \|\D_j\p\|_{L^2}^2).\]  

From the above inequality and using Berntein's lemma, we verify that\beno
\begin{split}
&\f{d}{dt}\big(\|\D_ju\|_{L^2}+\|\D_j\p\|_{L^2}+\|\D_j\vf\|_{L^2}\big)+2^{2j}\|\D_ju\|_{L^2}+2^{2j} \|\D_j \p\|_{L^2}\\
\lesssim &(1+2^j)\big(\|f_j\|_{L^2}+(1+2^j)\|g_j\|_{L^2}+\|[u\cdot \na,\Lam]\D_j\p\|_{L^2}\big).
\end{split}
\eeno
Integrating in time and multiplying both sides of the above inequality by $2^{
\f12j}$, and summing up about $j\leq N$, we have
\beno
\begin{split}
&\|u^{\ell}\|_{\widetilde L^\infty_t(\dot B^{\f 12}_{2,1})}+\|\p^{\ell}\|_{\widetilde L^\infty_t(\dot B^{\f 12}_{2,1})}+\|\vf^{\ell}\|_{\widetilde L^\infty_t(\dot B^{\f 12}_{2,1})}+\|u^{\ell}\|_{ L^1_t(\dot B^{\f 52}_{2,1})}+\| \p^{\ell}\|_{ L^1_t(\dot B^{\f 52}_{2,1})}\\
\lesssim &\|u_0^{\ell}\|_{\dot B^{\f 12}_{2,1}}+\|\p^{\ell}_0\|_{\dot B^{\f 12}_{2,1}}+\|\vf^{\ell}_0\|_{\dot B^{\f 12}_{2,1}}+\int_0^t\sum_{j\leq N}2^{\f12j}\big(\|f_j\|_{L^2}+\|g_j\|_{L^2}+\|[u\cdot \na,\Lam]\D_j\p\|_{L^2}\big)ds.
\end{split}
\eeno
Thanks to Lemma \ref{com} and Lemma \ref{mult}, we verify that
\beno
\sum_{j\leq N}2^{\f12j}\|f_j\|_{L^2}
\lesssim \big(\|u^{\ell}\|_{\dot B^{\f 12}_{2,1}}+\|u^h\|_{\dot B^{\f 3p-1}_{p,1}}\big)\big(\|\na u^{\ell}\|_{\dot B^{\f 32}_{2,1}}+\|\na u^h\|_{\dot B^{\f 3p}_{p,1}}\big),
\eeno
\beno
\begin{split}
\sum_{j\leq N}2^{\f12j}\|g_j\|_{L^2}
\lesssim&  \big(\|\s^{\ell}\|_{\dot B^{\f 12}_{2,1}}+\|\s^h\|_{\dot B^{\f 3p-1}_{p,1}}\big)\big(\|\na u^{\ell}\|_{\dot B^{\f 32}_{2,1}}+\|\na u^h\|_{\dot B^{\f 3p}_{p,1}}+\|\tr\s^{\ell}\|_{\dot B^{\f 32}_{2,1}}+\|\tr\s^h\|_{\dot B^{\f 3p}_{p,1}}
\big),
\end{split}
\eeno
\beno
\begin{split}
\sum_{j\leq N}2^{\f12j}\|[u\cdot \na,\Lam]\D_j\p\|_{L^2}
\lesssim&\sum_{j\leq N}2^{\f12j}\|\na u\|_{L^\infty}\|\na\D_j\p\|_{L^2}\\
\lesssim&\|\p^{\ell}\|_{\dot B^{\f 32}_{2,1}}\|\na u\|_{\dot B^{\f 3p}_{p,1}}
\lesssim\|\s^{\ell}\|_{\dot B^{\f 12}_{2,1}}(\|\na u^{\ell}\|_{\dot B^{\f 32}_{2,1}}
+\|\na u^h\|_{\dot B^{\f 3p}_{p,1}}).
\end{split}
\eeno
Combining the above estimates yields that\beq\label{e2}
\begin{split}
&\|u^{\ell}\|_{\widetilde L^\infty_t(\dot B^{\f 12}_{2,1})}+\|\p^{\ell}\|_{\widetilde L^\infty_t(\dot B^{\f 12}_{2,1})}+\|u^{\ell}\|_{ L^1_t(\dot B^{\f 52}_{2,1})}+\| \p^{\ell}\|_{ L^1_t(\dot B^{\f 52}_{2,1})}\\
\lesssim &\mathcal E(0)+\int_0^t\big(\|u^{\ell}\|_{\dot B^{\f 12}_{2,1}}+\|u^h\|_{\dot B^{\f 3p-1}_{p,1}}+\|\s^{\ell}\|_{\dot B^{\f 12}_{2,1}}+\|\s^h\|_{\dot B^{\f 3p}_{p,1}}\big)\\
&\quad\times\big(\|\na u^{\ell}\|_{\dot B^{\f 32}_{2,1}}+\|\na u^h\|_{\dot B^{\f 3p}_{p,1}}+\|\tr\s^{\ell}\|_{\dot B^{\f 32}_{2,1}}+\|\tr\s^h\|_{\dot B^{\f 3p}_{p,1}}\big)ds,
\end{split}
\eeq
which implies\beq
\mathcal E_2(t)
\lesssim \mathcal E(0)+\big(\mathcal E_1(t)+\mathcal E_3(t)\big)\big(\mathcal E_2(t)+\mathcal E_3(t)+\mathcal E_4(t)\big).
\eeq

\subsection{The estimates of $\mathcal E_3(t)$}
Denote $\G=u-\Lam^{-1}\p$, we can get from \eqref{ob3} that
\beq\label{ob4}
\D_j\G_{t}+u\cdot \na \D_j\G+\f{1}{\f1{c_0}+t}\D_j\G-\tri \D_j\G=(\f{1}{\f1{c_0}+t}-\f12) \D_j u+[u\cdot\na,\Lam^{-1}]\D_j\p+f_j-\Lam^{-1}g_j.
\eeq
By the standard $L^p$ estimate, we get\beno
\f{d}{dt}\|\D_j\G\|_{L^p}+2^{2j}\|\D_j\G\|_{L^p}
\lesssim\|\D_j u\|_{L^p}+\|[u\cdot\na,\Lam^{-1}]\D_j\p\|_{L^p}+\|f_j \|_{L^p}+\|\Lam^{-1}g_j\|_{L^p}.
\eeno
Integrating in time and multiplying both sides of the above inequality by $2^{(\f 3p-1)j}$, summing up about $j> N$, then we obtain that
\beno
\begin{split}
&\|\G^h\|_{\widetilde L^\infty_t(\dot B^{\f 3p-1}_{p,1})}+\|\G^h\|_{ L^1_t(\dot B^{\f 3p+1}_{p,1})}
\lesssim\|\G^h_0\|_{\dot B^{\f 3p-1}_{p,1}}+2^{-2N}\big(\|\G^h\|_{ L^1_t(\dot B^{\f 3p+1}_{p,1})}+\|\p^h\|_{L^1_t(\dot B^{\f 3p}_{p,1})}\big)\\
&\quad+\int_0^t\sum_{j>N} 2^{(\f 3p-1)j}\big(\|[u\cdot\na,\Lam^{-1}]\D_j\p\|_{L^p}+\|f_j \|_{L^p}+\|\Lam^{-1}g_j\|_{L^p}\big)ds.
\end{split}
\eeno
Lemma \ref{com} and Lemma \ref{mult} ensure that\beno
\begin{split}
\sum_{j>N} 2^{(\f 3p-1)j}\|[u\cdot\na,\Lam^{-1}]\D_j\p\|_{L^p}
\lesssim& \sum_{j>N} 2^{\f 3pj}\|\Lam^{-1}u\|_{L^\infty}\|\D_j\p\|_{L^p}\\
\lesssim &\|u\|_{\dot B^{\f 3p-1}_{p,1}}\|\p^h\|_{\dot B^{\f 3p}_{p,1}}
\lesssim \big(\|u^{\ell}\|_{\dot B^{\f 12}_{2,1}}+\|u^h\|_{\dot B^{\f 3p-1}_{p,1}}\big)\|\p^h\|_{\dot B^{\f 3p}_{p,1}},
\end{split}
\eeno
\beno
\sum_{j>N} 2^{(\f 3p-1)j}\|f_j \|_{L^p}
\lesssim\big(\|u^{\ell}\|_{\dot B^{\f 12}_{2,1}}+\|u^h\|_{\dot B^{\f 3p-1}_{p,1}}\big)\big(\|\na u^{\ell}\|_{\dot B^{\f 32}_{2,1}}+\|\na u^h\|_{\dot B^{\f 3p}_{p,1}}\big),
\eeno
\beno
\begin{split}
\sum_{j>N} 2^{(\f 3p-1)j}\|\Lam^{-1}g_j\|_{L^p}
\lesssim&\big(\|\s^{\ell}\|_{\dot B^{\f 12}_{2,1}}+\|\s^h\|_{\dot B^{\f 3p-1}_{p,1}}\big)\big(\|\na u^{\ell}\|_{\dot B^{\f 32}_{2,1}}+\|\na u^h\|_{\dot B^{\f 3p}_{p,1}}+\|\tr\s^{\ell}\|_{\dot B^{\f 32}_{2,1}}+\|\tr\s^h\|_{\dot B^{\f 3p}_{p,1}}\big).
\end{split}
\eeno
Together with the above estimates, we deduce that
\beq\label{G}
\begin{split}
&\|\G^h\|_{\widetilde L^\infty_t(\dot B^{\f 3p-1}_{p,1})}+\|\G^h\|_{ L^1_t(\dot B^{\f 3p+1}_{p,1})}\\
\lesssim&\|\G^h_0\|_{\dot B^{\f 3p-1}_{p,1}}+2^{-2N}\big(\|\G^h\|_{ L^1_t(\dot B^{\f 3p+1}_{p,1})}+\|\p^h\|_{L^1_t(\dot B^{\f 3p}_{p,1})}\big)\\
&\quad+\int_0^t\big(\|u^{\ell}\|_{\dot B^{\f 12}_{2,1}}+\|u^h\|_{\dot B^{\f 3p-1}_{p,1}}+\|\s^{\ell}\|_{\dot B^{\f 12}_{2,1}}+\|\s^h\|_{\dot B^{\f 3p}_{p,1}}\big)\\
&\qquad\times\big(\|\p^h\|_{\dot B^{\f 3p}_{p,1}}+\|\na u^{\ell}\|_{\dot B^{\f 32}_{2,1}}+\|\na u^h\|_{\dot B^{\f 3p}_{p,1}}+\|\tr\s^{\ell}\|_{\dot B^{\f 32}_{2,1}}+\|\tr\s^h\|_{\dot B^{\f 3p}_{p,1}}\big)ds.
\end{split}
\eeq
We rewrite the second equation of \eqref{ob3} as follows\beq\label{ob5}
\D_j \p_t+u\cdot \na \D_j\p+ (\f12+\f{1}{\f1{c_0}+t}) \D_j \p=g_j-\f12\Lam\D_j\G.
\eeq
By virtue of the standard $L^p$ estimate, we get\beno
\f{d}{dt}\|\D_j\p\|_{L^p}+\|\D_j\p\|_{L^p}
\lesssim\|g_j\|_{L^p}+\|\Lam\D_j\G\|_{L^p},
\eeno
which leads to
\beq\label{p}
\begin{split}
&\|\p^h\|_{\widetilde L^\infty_t(\dot B^{\f 3p}_{p,1})}+\|\p^h\|_{L^1_t(\dot B^{\f 3p}_{p,1})}\\
\lesssim&\|\p^h_0\|_{\dot B^{\f 3p}_{p,1}}+\int_0^t\sum_{j>N} 2^{\f 3pj}\big(\|g_j\|_{L^p}+\|\Lam\D_j\G\|_{L^p}\big)ds\\
\lesssim&\|\p^h_0\|_{\dot B^{\f 3p}_{p,1}}+\|\G^h\|_{ L^1_t(\dot B^{\f 3p+1}_{p,1})}+\int_0^t\sum_{j>N} 2^{\f 3pj}\|g_j\|_{L^p}ds\\
\lesssim&\|\p^h_0\|_{\dot B^{\f 3p}_{p,1}}+\|\G^h\|_{ L^1_t(\dot B^{\f 3p+1}_{p,1})}+\int_0^t
\big(\|\s^{\ell}\|_{\dot B^{\f 32}_{2,1}}+\|\s^h\|_{\dot B^{\f 3p}_{p,1}}\big)\\
&\quad\times\big(\|\na u^{\ell}\|_{\dot B^{\f 32}_{2,1}}+\|\na u^h\|_{\dot B^{\f 3p}_{p,1}}+\|\tr\s^{\ell}\|_{\dot B^{\f 32}_{2,1}}+\|\tr\s^h\|_{\dot B^{\f 3p}_{p,1}}\big)ds.
\end{split}
\eeq
Taking $N\in\mathbb N^+$ is large enough, together with \eqref{G} and \eqref{p}, we verify that\beno
\begin{split}
&\|\G^h\|_{\widetilde L^\infty_t(\dot B^{\f 3p-1}_{p,1})}+\|\p^h\|_{\widetilde L^\infty_t(\dot B^{\f 3p}_{p,1})}+\|\G^h\|_{ L^1_t(\dot B^{\f 3p+1}_{p,1})}
+\|\p^h\|_{L^1_t(\dot B^{\f 3p}_{p,1})}\\
\lesssim &\mathcal E(0)
+\int_0^t\big (\|u^{\ell}\|_{\dot B^{\f 12}_{2,1}}+\|u^h\|_{\dot B^{\f 3p-1}_{p,1}}+\|\s^{\ell}\|_{\dot B^{\f 12}_{2,1}}+\|\s^h\|_{\dot B^{\f 3p}_{p,1}} \big)\\
&\quad\times\big(\|\p^h\|_{\dot B^{\f 3p}_{p,1}}+\|\na u^{\ell}\|_{\dot B^{\f 32}_{2,1}}+\|\na u^h\|_{\dot B^{\f 3p}_{p,1}}+\|\tr\s^{\ell}\|_{\dot B^{\f 32}_{2,1}}+\|\tr\s^h\|_{\dot B^{\f 3p}_{p,1}}\big)ds.
\end{split}
\eeno
We can deduce from $u=\G+\Lam^{-1}\p$ that\beq\label{e3}
\begin{split}
&\|u^h\|_{\widetilde L^\infty_t(\dot B^{\f 3p-1}_{p,1})}+\|\s^h\|_{\widetilde L^\infty_t(\dot B^{\f 3p}_{p,1})}+\|u^h\|_{L^1_t(\dot B^{\f 3p+1}_{p,1})}+\|\p^h\|_{L^1_t(\dot B^{\f 3p}_{p,1})}\\
\lesssim&\|\G^h\|_{\widetilde L^\infty_t(\dot B^{\f 3p-1}_{p,1})}+\|\p^h\|_{\widetilde L^\infty_t(\dot B^{\f 3p}_{p,1})}+\|\G^h\|_{ L^1_t(\dot B^{\f 3p+1}_{p,1})}
+\|\p^h\|_{L^1_t(\dot B^{\f 3p}_{p,1})},
\end{split}
\eeq
which implies\beq
\mathcal E_3(t)
\lesssim \mathcal E(0)+\big(\mathcal E_1(t)+\mathcal E_3(t)\big)\big(\mathcal E_2(t)+\mathcal E_3(t)+\mathcal E_4(t)\big).
\eeq

\subsection{The estimates of $\mathcal E_4(t)$}
Applying $\tr$ to the second equation of \eqref{PTT}, yields that\beno
(\tr \s)_t+u\cdot\na \tr \s+(\tr\s)^2+\f2{\f1{c_0}+t}\tr\s=0,
\eeno
which leads to\beno
\big[{(\f1{c_0}+t)^2}\tr\s\big]_t+u\cdot\na\big[(\f1{c_0}+t)^2\tr\s\big]
\leq-(\f1{c_0}+t)^2(\tr\s)^2.
\eeno
Applying the operator $\D_j$ to above equation, we get\beno
\big[(\f1{c_0}+t)^2\D_j\tr\s\big]_t+u\cdot\na \big[(\f1{c_0}+t)^2\D_j\tr\s\big]
\leq-(\f1{c_0}+t)^2\big(\D_j(\tr\s)^2-[u\cdot\na,\D_j]\tr\s\big).
\eeno
By virtue of the standard $L^q$ estimate, we obtain\beno
\f{d}{dt}\big[(\f1{c_0}+t)^2\|\D_j\tr\s\|_{L^q}\big]
\lesssim(\f1{c_0}+t)^2\big(\|\D_j(\tr\s)^2\|_{L^q} +\|[u\cdot\na,\D_j]\tr\s\|_{L^q}\big).
\eeno
 Choose $q=2$ and $q=p$ with respectively. Multiplying both sides of the above inequality by $2^{\f 32j}$ and $2^{\f 3pj}$, and then summing up about $j\leq N$ and $j> N$ with respectively, then we obtain that
\beno
\begin{split}
&\f{d}{dt}\Big[(\f1{c_0}+t)^2\big(\|\tr\s^{\ell}\|_{\dot B^{\f 32}_{2,1}}+\|\tr\s^h\|_{\dot B^{\f 3p}_{p,1}})\Big]\\
\lesssim
&\Big[(\f1{c_0}+t)^2\big(\|\tr\s^{\ell}\|_{\dot B^{\f 32}_{2,1}}+\|\tr\s^h\|_{\dot B^{\f 3p}_{p,1}}\big)\Big]\big(\|\tr\s^{\ell}\|_{\dot B^{\f 32}_{2,1}}+\|\tr\s^h\|_{\dot B^{\f 3p}_{p,1}}+\|\na u^{\ell}\|_{\dot B^{\f 32}_{2,1}}+\|\na u^h\|_{\dot B^{\f 3p}_{p,1}}\big).
\end{split}
\eeno
By Gronwall's inequality, then we have
\beno
\begin{split}
&(\f1{c_0}+t)^2\big(\|\tr\s^{\ell}\|_{\dot B^{\f 32}_{2,1}}+\|\tr\s^h\|_{\dot B^{\f 3p}_{p,1}}\big)\\
\lesssim&
\f1{c_0}\big(\|\tr\s^{\ell}_0\|_{\dot B^{\f 32}_{2,1}}+\|\tr\s^h_0\|_{\dot B^{\f 3p}_{p,1}}\big)
\exp\Big\{\int_0^t\big(\|\tr\s^{\ell}\|_{\dot B^{\f 32}_{2,1}}+\|\tr\s^h\|_{\dot B^{\f 3p}_{p,1}}+\|\na u^{\ell}\|_{\dot B^{\f 32}_{2,1}}+\|\na u^h\|_{\dot B^{\f 3p}_{p,1}}\big)\Big\},
\end{split}
\eeno
which leads to
\beq
\begin{split}
&\int_0^t\big(\|\tr\s^{\ell}\|_{\dot B^{\f 32}_{2,1}}+\|\tr\s^h\|_{\dot B^{\f 3p}_{p,1}}\big)ds\\
\lesssim&\f1{c_0}\big(\|\s^{\ell}_0\|_{\dot B^{\f 12}_{2,1}}+\|\s^h_0\|_{\dot B^{\f 3p}_{p,1}}\big)\\
&\quad\times\exp\Big\{\int_0^t\big(\|\tr\s^{\ell}\|_{\dot B^{\f 32}_{2,1}}+\|\tr\s^h\|_{\dot B^{\f 3p}_{p,1}}+\|\na u^{\ell}\|_{\dot B^{\f 32}_{2,1}}+\|\na u^h\|_{\dot B^{\f 3p}_{p,1}}\big)ds\Big\},
\end{split}
\eeq
that is
\beq
\mathcal E_4(t)
\lesssim {\mathcal E(0)} \exp\big(\mathcal E_2(t)+\mathcal E_3(t)+\mathcal E_4(t)\big).
\eeq

\subsection{Proof of the Theorem \ref{them}}
\begin{proof}
In this subsection, we will combine the above a priori estimates of $\mathcal E_1(t)$, $\mathcal E_2(t)$, $\mathcal E_3(t)$ and $\mathcal E_4(t)$ together and give the proof of the Theorem \ref{them}, then exists for any $t\in[0,T]$, we have\beq\label{lam}
\begin{split}
E(t)=&\mathcal E_1(t)+\mathcal E_2(t)+\mathcal E_3(t)+\mathcal E_4(t)\\
\leq&C^*\Big[1+\exp\big(\mathcal E_2(t)+\mathcal E_3(t)+\mathcal E_4(t))\big)\Big]\mathcal E(0)+\Big[\mathcal E_1^2(t)+\mathcal E_2^2(t)+\mathcal E_3^2(t)+\mathcal E_4^2(t)\Big].
\end{split}
\eeq
Due to the local existence theory, there exists a positive time $T$ such that\beq\label{lambda}
E(t) \leq 3C^*\d_0,\quad \forall t\in[0,T].
\eeq
Let $T^*$ be the largest possible time of $T$ for what \eqref{lambda} holds. Under the setting of initial data, there exists a small enough number $\ep_0$ such that $\mathcal E(0)\leq \d_0\leq \ep_0$. By virtue of \eqref{lam} and the smallness assumption on $\d_0$ , we get that\beno
E(t)\leq 2C^*\d_0+C\d_0^2<3C^*\d_0.
\eeno
 By standard continuity argument and total energy \eqref{lam}, we can show that $T^*=\infty$ provided that $\d_0$ is small enough. Hence, we finish the proof of the Theorem \ref{them}
\end{proof}

\begin{rema}
As $\f {1}{\f 1c_0+t}$ does not belong to any Besov spaces,  we can't  say that $\tau$ belongs to any Besov spaces. However, we have
\[\|\s\|_{L^\infty_{t,x}}\lesssim \|\s^\ell \|_{L^\infty_{t,x}}+\|\s^h\|_{L^\infty_{t,x}}\lesssim  \|\s^\ell \|_{L^\infty_{t}({\dot B^{\f 32}_{2,1}})}+\|\s^h\|_{L^\infty_{t}(\dot B^{\f 3p}_{p,1})}\lesssim  \|\s^\ell \|_{L^\infty_{t}({\dot B^{\f 12}_{2,1}})}+\|\s^h\|_{L^\infty_{t}(\dot B^{\f 3p}_{p,1})},\]
which implies that $\tau\in L^\infty_{t,x}$. Moreover, one can check that $\tau\in C(\mathbb{R}^+\times \mathbb{R}^3)$
\end{rema}

\begin{rema}
The initial condition $c_0>0$ and $\ep_0$ is small enough implies that $\inf \tr \tau_0>0$. On the other hand, if there exists a $x_0\in\mathbb{R}^3$ such that $\tr \tau_0(x_0)<0$, we can deduce from the second equation of \eqref{OB} that
\beq
\tr \tau_t+u\cdot \nabla \tr \tau+(\tr \tau)^2=0.
\eeq
Consider the trajectory equation
\beno
\f{d}{dt} q(t,x)=u(t, q(t,x)),\quad q(0,x)=x.
\eeno
It is easy to see that
\beq
\tr \tau(t,q(t,x_0))=\f{\tr \tau_0(x_0)}{1+\tr\tau_0(x_0)t}, \quad \forall t\in[0,T],
\eeq
which leads to $\tr \tau(t, q(t,x_0))$ blows up in finite time. Thus, the condition $\inf \tr \tau_0\geq 0$ is a necessary condition to ensure that the strong solution exists globally.
\end{rema}

{\bf Acknowledgements}.
Wei Luo is partially supported by NSF of China under Grant 11701586 and 11671407. Xiaoping Zhai is partially supported by NSF of China under Grant 11601533.

\bibliographystyle{abbrv} 
\bibliography{PTTref}

\begin{thebibliography}{10}

\bibitem{Bahouri2011}
H.~Bahouri, J.-Y. Chemin, and R.~Danchin.
\newblock {\em Fourier analysis and nonlinear partial differential equations},
  volume 343 of {\em Grundlehren der Mathematischen Wissenschaften}.
\newblock Springer, Heidelberg, 2011.

\bibitem{Bautista}
O.~Bautista, S.~S\'{a}nchez, J.~C. Arcos, and F.~M\'{e}ndez.
\newblock Lubrication theory for electro-osmotic flow in a slit microchannel
  with the {P}han-{T}hien and {T}anner model.
\newblock {\em J. Fluid Mech.}, 722:496--532, 2013.

\bibitem{Bird1977}
R.~B. Bird, R.~C. Armstrong, and O.~Hassager.
\newblock {\em Dynamics of Polymeric Liquids}, volume~1.
\newblock Wiley, New York, 1977.

\bibitem{Chemin2001}
J.-Y. Chemin and N.~Masmoudi.
\newblock About lifespan of regular solutions of equations related to
  viscoelastic fluids.
\newblock {\em SIAM J. Math. Anal.}, 33(1):84--112, 2001.

\bibitem{Chen-Hao}
Q.~Chen and X.~Hao.
\newblock Global well-posedness in the critical {B}esov spaces for the
  incompressible {O}ldroyd-{B} model without damping mechanism.
\newblock {\em arXiv:1810.08048}, 2018.

\bibitem{Chen-Miao}
Q.~Chen, C.~Miao, and Z.~Zhang.
\newblock Global well-posedness for compressible {N}avier-{S}tokes equations
  with highly oscillating initial velocity.
\newblock {\em Comm. Pure Appl. Math.}, 63(9):1173--1224, 2010.

\bibitem{Chen-Zhang}
Q.~Chen, C.~Miao, and Z.~Zhang.
\newblock Well-posedness in critical spaces for the compressible
  {N}avier-{S}tokes equations with density dependent viscosities.
\newblock {\em Rev. Mat. Iberoam.}, 26(3):915--946, 2010.

\bibitem{Chen-Miao-Zhang}
Q.~Chen, C.~Miao, and Z.~Zhang.
\newblock On the ill-posedness of the compressible {N}avier-{S}tokes equations
  in the critical {B}esov spaces.
\newblock {\em Rev. Mat. Iberoam.}, 31(4):1375--1402, 2015.

\bibitem{Danchin2000}
R.~Danchin.
\newblock Global existence in critical spaces for compressible
  {N}avier-{S}tokes equations.
\newblock {\em Invent. Math.}, 141(3):579--614, 2000.

\bibitem{Danchin2001}
R.~Danchin.
\newblock Global existence in critical spaces for flows of compressible viscous
  and heat-conductive gases.
\newblock {\em Arch. Ration. Mech. Anal.}, 160(1):1--39, 2001.

\bibitem{Danchin-He}
R.~Danchin and L.~He.
\newblock The incompressible limit in {$L^p$} type critical spaces.
\newblock {\em Math. Ann.}, 366(3-4):1365--1402, 2016.

\bibitem{Fang-Zi2013}
D.~Fang, M.~Hieber, and R.~Zi.
\newblock Global existence results for {O}ldroyd-{B} fluids in exterior
  domains: the case of non-small coupling parameters.
\newblock {\em Math. Ann.}, 357(2):687--709, 2013.

\bibitem{Fang-Zi2016}
D.~Fang and R.~Zi.
\newblock Global solutions to the {O}ldroyd-{B} model with a class of large
  initial data.
\newblock {\em SIAM J. Math. Anal.}, 48(2):1054--1084, 2016.

\bibitem{Fernandez-Cara}
E.~Fern\'andez-Cara, F.~Guill\'en, and R.~R. Ortega.
\newblock Some theoretical results concerning non-{N}ewtonian fluids of the
  {O}ldroyd kind.
\newblock {\em Ann. Scuola Norm. Sup. Pisa Cl. Sci. (4)}, 26(1):1--29, 1998.

\bibitem{Tamaddon-Jahromi}
I.~E. Gardu\~{n}o, H.~R. Tamaddon-Jahromi, K.~Walters, and M.~F. Webster.
\newblock The interpretation of a long-standing rheological flow problem using
  computational rheology and a {PTT} constitutive model.
\newblock {\em J. Non-Newton. Fluid Mech.}, 233:27--36, 2016.

\bibitem{Guillope1990-NA}
C.~Guillop\'e and J.-C. Saut.
\newblock Existence results for the flow of viscoelastic fluids with a
  differential constitutive law.
\newblock {\em Nonlinear Anal.}, 15(9):849--869, 1990.

\bibitem{Guillope1990}
C.~Guillop\'e and J.-C. Saut.
\newblock Global existence and one-dimensional nonlinear stability of shearing
  motions of viscoelastic fluids of {O}ldroyd type.
\newblock {\em RAIRO Mod\'el. Math. Anal. Num\'er.}, 24(3):369--401, 1990.

\bibitem{Lei2008}
Z.~Lei, C.~Liu, and Y.~Zhou.
\newblock Global solutions for incompressible viscoelastic fluids.
\newblock {\em Arch. Ration. Mech. Anal.}, 188(3):371--398, 2008.

\bibitem{Lei-Zhou2005}
Z.~Lei and Y.~Zhou.
\newblock Global existence of classical solutions for the two-dimensional
  {O}ldroyd model via the incompressible limit.
\newblock {\em SIAM J. Math. Anal.}, 37(3):797--814, 2005.

\bibitem{Lin-Liu-Zhang2005}
F.-H. Lin, C.~Liu, and P.~Zhang.
\newblock On hydrodynamics of viscoelastic fluids.
\newblock {\em Comm. Pure Appl. Math.}, 58(11):1437--1471, 2005.

\bibitem{Lions-Masmoudi}
P.~L. Lions and N.~Masmoudi.
\newblock Global solutions for some {O}ldroyd models of non-{N}ewtonian flows.
\newblock {\em Chinese Ann. Math. Ser. B}, 21(2):131--146, 2000.

\bibitem{Mu-Zhao2013}
Y.~Mu, G.~Zhao, A.~Chen, and X.~Wu.
\newblock Modeling and simulation of three-dimensional extrusion swelling of
  viscoelastic fluids with {PTT}, {G}iesekus and {FENE}-{P} constitutive
  models.
\newblock {\em Internat. J. Numer. Methods Fluids}, 72(8):846--863, 2013.

\bibitem{Mu-Zhao2012}
Y.~Mu, G.~Zhao, X.~Wu, and J.~Zhai.
\newblock Modeling and simulation of three-dimensional planar contraction flow
  of viscoelastic fluids with {PTT}, {G}iesekus and {FENE}-{P} constitutive
  models.
\newblock {\em Appl. Math. Comput.}, 218(17):8429--8443, 2012.

\bibitem{Oldroyd}
J.~G. Oldroyd.
\newblock Non-{N}ewtonian effects in steady motion of some idealized
  elastico-viscous liquids.
\newblock {\em Proc. Roy. Soc. London. Ser. A}, 245:278--297, 1958.

\bibitem{Oliveira1999}
P.~J. Oliveira and F.~T. Pinho.
\newblock Analytical solution for fully developed channel and pipe flow of
  {P}han-{T}hien--{T}anner fluids.
\newblock {\em J. Fluid Mech.}, 387:271--280, 1999.

\bibitem{Thien1978}
N.~Phan-Thien.
\newblock A nonlinear network viscoelastic model.
\newblock {\em Journal of Rheology}, 22(3):259--283, 1978.

\bibitem{Thien1977}
N.~Phan-Thien and R.~I. Tanner.
\newblock A new constitutive equation derived from network theory.
\newblock {\em Journal of Non-Newtonian Fluid Mechanics}, 2(4):353 -- 365,
  1977.

\bibitem{Zhai}
X.~Zhai.
\newblock Global solutions to the $n$-dimensional incompressible {O}ldroyd-{B}
  model without damping mechanism.
\newblock {\em arXiv:1810.08048}, 2018.

\bibitem{Zhang-Fang2012}
T.~Zhang and D.~Fang.
\newblock Global existence of strong solution for equations related to the
  incompressible viscoelastic fluids in the critical {$L^p$} framework.
\newblock {\em SIAM J. Math. Anal.}, 44(4):2266--2288, 2012.

\bibitem{Zhu2018}
Y.~Zhu.
\newblock Global small solutions of 3{D} incompressible {O}ldroyd-{B} model
  without damping mechanism.
\newblock {\em J. Funct. Anal.}, 274(7):2039--2060, 2018.

\end{thebibliography}

\end{document}